\documentclass[a4paper,12pt]{amsart}
\usepackage{amsmath,amssymb,amsfonts,amsthm}
\usepackage{latexsym,mathrsfs}
\usepackage{graphicx}
\usepackage[all]{xy}
\input xypic
\usepackage{color}
\usepackage[pagebackref]{hyperref}
\hypersetup{colorlinks=true,linkcolor=red,citecolor=blue}
\usepackage[OT2,T1]{fontenc}
\usepackage[utf8]{inputenc}
\usepackage{lmodern}
\usepackage{microtype}
\usepackage{enumerate}
\usepackage{tikz-cd}
\usepackage{bm}
\usepackage{mathtools}
\usepackage{tensor}
\theoremstyle{plain}
\newtheorem{theorem}[subsection]{{\bf Theorem}}

\newtheorem{corollary}[subsection]{{\bf Corollary}}
\newtheorem{proposition}[subsection]{{\bf Proposition}}
\newtheorem{lemma}[subsection]{{\bf Lemma}}
\theoremstyle{definition}
\newtheorem*{definition}{{\bf Definition}}
\theoremstyle{remark}

\newtheorem{example}[subsection]{\emph{Example}}
\numberwithin{equation}{subsection}
\DeclareMathOperator{\im}{im}

\DeclareMathOperator{\B}{B}

\DeclareMathOperator{\GL}{GL}

\DeclareMathOperator{\Gr}{Gr}

\DeclareBoldMathCommand{\bbot}{\bot}

\DeclareSymbolFont{cyrletters}{OT2}{wncyr}{m}{n}
\DeclareMathSymbol{\Sha}{\mathalpha}{cyrletters}{"58}

\newcommand{\Z}{\mathbb{Z}}
\newcommand{\N}{\mathbb{Z}_{\ge 0}}
\newcommand{\Q}{\mathbb{Q}}

\newcommand{\cwedge}{\curlywedge}

\begin{document}
\title[Bogomolov multipliers]{Computing Bogomolov multipliers and finitary Bogomolov multipliers of infinite groups}
\author{Primo\v z Moravec}
\address{{
Faculty of  Mathematics and Physics, University of Ljubljana,
and Institute of Mathematics, Physics and Mechanics,
Slovenia}}
\email{primoz.moravec@fmf.uni-lj.si}
\thanks{ORCID: \url{https://orcid.org/0000-0001-8594-0699}. The author acknowledges the financial support from the Slovenian Research and Innovation Agency (ARIS), research core funding No. P1-0222, and project No. J1-50001.}
\date{today}
\begin{abstract}
\noindent
We study homological Bogomolov multipliers and finitary Bogomolov multipliers of infinite groups, with emphasis on polycyclic groups. We obtain structural descriptions for torsion-free nilpotent groups of class two, semidirect products, wreath products, and infinite simple groups. The class-two case reveals a connection with rational points on linear sections of Grassmannians, explaining an arithmetic obstruction to a general algorithm. On the positive side, we give a deterministic algorithm for virtually abelian polycyclic groups and an algorithm for finitary Bogomolov multipliers of arbitrary polycyclic groups, both implemented in GAP.
\end{abstract}
\maketitle


\section{Introduction}
\label{s:intro}

\noindent
Let $G$ be a group. A result of Miller \cite{Mil52} relates the Schur multiplier
$M(G)=H_2(G,\mathbb Z)$ of $G$ with the universal relations among commutators in $G$,  Inside $M(G)$
lies the subgroup $M_0(G)$
generated by the elementary relations arising from commuting pairs.  The
quotient $B_0(G)=M(G)/M_0(G)$
is the homological Bogomolov multiplier.  Its cohomological dual was
introduced by Bogomolov in connection with unramified Brauer groups and
Noether's problem \cite{Bog88}. The homological description, valid for
arbitrary groups, was developed in \cite{Mor12}.  Thus $B_0(G)$ may be
viewed as the group of universal commutator relations which are not
generated by the evident relations attached to commuting elements.

For finite groups this invariant is effectively computable: one may
construct $M(G)$ and enumerate the finitely many commuting pairs needed
to generate $M_0(G)$.  For an infinite polycyclic group, the algorithms
of Eick and Nickel \cite{Eic08} still compute $M(G)$.  
The new difficulty is that the commuting locus is infinite,
and an effective presentation of $M(G)$ does not by itself provide a
finite generating set for $M_0(G)$.  

The main purpose of this paper is to
study this obstruction, to isolate classes in which it can be overcome,
and to give exact algorithms where possible. The cases include torsion-free nilpotent groups of class 2, virtually abelian groups, and split extensions of groups. We also compute the finitary version of Bogomolov multipliers for arbitrary polycyclic groups.

For torsion-free nilpotent groups of class two with torsion-free
abelianization, the problem admits a concrete linear-algebraic
description.  Writing
$A=G/G'\cong\mathbb Z^d$, and $C=G'\cong\mathbb Z^c$,
and denoting the commutator map by
$\beta\colon\wedge^2A\rightarrow C$,
we first observe that
$B_0(G)\cong L_\beta/D_\beta$, where
$L_\beta=\ker\beta$, and 
$D_\beta=\langle a\wedge b\in L_\beta\rangle$. This formula is reminiscent of Bogomolov's description of $B_0$ for finite $p$-groups of class 2  \cite{Bog88}.
The subgroup $D_\beta$ is governed by the rational points of the
Grassmannian section
$
 X_{L_\beta}
   =\mathbb P(L_\beta\otimes\mathbb Q)
      \cap\operatorname{Gr}(2,A\otimes\mathbb Q).
$ A similar observation has been applied by Jezernik and S'anchez \cite{JS24}.
Consequently, a uniform terminating algorithm computing even the rank of
$L_\beta/D_\beta$ would decide rational-point existence on all such
linear sections.  In dimension five these include genus-one normal
quintics, for which no unconditional uniform terminating rational-point
algorithm is presently known.  This explains why the apparently modest
task of generating commuting wedges already encounters genuine
arithmetic difficulties.  We also obtain effective low-dimensional
cases and explicit torsion-free class-two groups with nontrivial
Bogomolov multiplier.

Our positive algorithmic results are based on polynomial
parametrizations of commuting pairs.  We introduce the notion of a
polynomial cover of the commuting loci and prove that, whenever a
polycyclic group admits such a cover, finite-difference identities reduce
the infinitely many commuting wedges to finitely many explicitly
specified values.  For finitely generated virtually abelian groups the
commuting lifts above each pair in the finite quotient form an effectively
computable affine lattice, and the exterior-square pairing has degree at
most two on this lattice.  This yields a deterministic terminating
algorithm for $B_0(G)$ for every virtually abelian polycyclic group.  The
algorithm is implemented in GAP \cite{GAP4} using the package \textsf{Polycyclic}.

We also establish structural results beyond the virtually abelian case.
Using the Lyndon--Hochschild--Serre spectral sequence, we prove that
$B_0(G)=0$ for every abelian-by-cyclic group, without a finiteness
assumption. Note that Bogomolov \cite{Bog88} previously proved this result for finite groups. More generally, Barge \cite{Bar89} showed that
$B_0(A\rtimes Q)=0$
whenever $A$ is finite abelian and $Q$ is finite bicyclic. Again, we show that the same result holds true without assuming any of the two groups being finite.

Bogomolov multipliers of semidirect products of two finite groups of coprime orders were described by Kang \cite{Kan14}.
For arbitrary split
extensions we obtain a relative decomposition of $B_0(G)$, and for
restricted regular wreath products a split right-exact sequence
$$
 B_0(H)\longrightarrow B_0(H\wr Q)
       \longrightarrow B_0(Q)\longrightarrow0.
$$
This exact sequence enables effective determination of $B_0(H\wr Q)$.

Kunyavskii \cite{Kun08} proved that finite simple groups have trivial Bogomolov multiplier. Our contribution here is a construction of countable
infinite simple groups $G$ with
$B_0(G)\cong\mathbb Z^{(\mathbb N)}$.

Finally, we study the finitary multiplier $B_{0F}(G)$, generated by
homology classes coming from finite subgroups.  We prove that it is finite
whenever $G$ has only finitely many conjugacy classes of finite subgroups,
and give an exact algorithm for polycyclic groups.  The accompanying GAP 
implementation shows that all crystallographic groups of dimensions at
most three have trivial finitary multiplier.  A twelve-dimensional
group with point group of order 64 provides an example with nontrivial finitary Bogomolov multiplier.


\section{Preliminaries}
\label{s:prelim}

\subsection{Schur multipliers}
Let $G$ be a group given by a free presentation $G=F/R$. By Hopf's formula \cite[Chapter~II, \S5, Theorem~5.3, pp.~42--43]{Bro82}, the second homology group $H_2(G,\Z)$ is isomorphic to $M(G)=(F'\cap R)/[F,R]$. The commutator subgroup of the group $\widehat{G}=F/[F,R]$ is isomorphic to the nonabelian exterior square $G\wedge G$ of Brown and Loday \cite{Bro87} via $[x,y][F,R]\mapsto xR\wedge yR$. Under this identification, given $g,h\in G$, if $\widehat{g}$ and $\widehat{h}$ are their arbitrary lifts in $\widehat{G}$, then $[\widehat{g},\widehat{h}]$ corresponds to $g\wedge h$. The group $M(G)$ is isomorphic to the kernel of the commutator map $\kappa:G\wedge G\twoheadrightarrow G'$.

\subsection{Bogomolov multipliers}
Let $K(F)$ be the set of all commutators in $F$. Set $M_0(G)=\langle K(F)\cap R\rangle/[F,R]$. Note that $M_0(G)\cong \langle x\wedge y\mid x,y\in G,[x,y]=1\rangle$, and we sometimes identify $M_0(G)$ with the right-hand side. Also, upon identifying $M(G)$ with $H_2(G,\Z)$ we have that $M_0(G)$ is isomorphic to the sum of all images of maps $H_2(A,\Z)\to H_2(G,\Z)$, as $A$ runs through all abelian subgroups of $G$, see \cite{Mor12}.

If $x,y\in G$ commute, the homomorphism $\Z^2\to G$ associated to this commuting pair induces a map $S^1\times S^1\rightarrow BG$. Let us denote the image of the oriented fundamental class by $\tau_G(x,y)\in H_2(G,\Z)$. In the normalized bar complex it is represented by $[x\mid y]-[y\mid x]$, so we have
$M_0(G)=\langle \tau_G(x,y)\mid x,y\in G,[x,y]=1\rangle$. The construction is functorial; if $f:G\rightarrow L$, then $H_2(f)\tau_G(x,y)=\tau_L(f(x),f(y))$.

The quotient $B_0(G)=M(G)/M_0(G)$ is the (homological) Bogomolov multiplier of $G$ \cite{Mor12}. It can be considered as a measure of the universal relations
among commutators in $G$ that are not generated by the elementary relations
arising from commuting pairs. Its cohomological dual has applications in Noether problem \cite{Bog88}. 

\subsection{Polycyclic groups}
If $G$ is a (not necessarily finite) polycyclic group given by a polycyclic presentation, then the algorithms of Eick and Nickel \cite{Eic08} enable effective calculations of polycyclic presentations of the groups $\widehat{G}$, $G\wedge G$, $M(G)$, as well as explicit description of the crossed pairing  $\lambda :G\times G\rightarrow G\wedge G$ mapping $(g,h)$ to $g\wedge h$. These functions are part of the GAP package \texttt{Polycyclic}.

An algorithm for computing the Bogomolov multipliers of finite polycyclic groups was developed in \cite{Mor12} and further optimized in \cite{Jez14}. It crucially relies on the fact that, in finite groups, one can effectively enumerate the commuting pairs and therefore generate $M_0(G)$. Computations of Bogomolov multipliers of general finite groups can be done using GAP package \texttt{HAP}.

\section{Torsion-free nilpotent groups of class 2}
\label{s:class2B0}

\noindent
Let $G$ be a finitely generated torsion-free nilpotent group of class 2 and assume that its abelianization is also torsion-free. Denote $A=G/G'\cong\Z^d$ and $C=G'\cong \Z^c$. The commutator map $G\times G\rightarrow C$ induces a surjective homomorphism 
$\beta:\wedge^2A\twoheadrightarrow C$.
Denote $L_\beta=\ker\beta$ and $D_\beta=\langle a\wedge b\mid \beta(a\wedge b)=0\rangle$. The following is essentially contained in \cite[Lemma 5.1]{Bog88}. We include a proof for completeness.

\begin{proposition}
    \label{prop:B0class2}
    Under the above notation, we have that $B_0(G)\cong L_\beta/D_\beta$.
\end{proposition}

\begin{proof}[Sketch of proof]
    Denote $G\cwedge G=(G\wedge G)/M_0(G)$. Note that $B_0(G)=\ker (G\cwedge G\twoheadrightarrow G')$.
    The map $A\times A\rightarrow G\cwedge G$ given by $(gG',hG')\mapsto g\cwedge h$ is well defined and biadditive, hence it induces an epimorphism $\phi:\wedge^2A\twoheadrightarrow G\cwedge G$. We have that $D_\beta\le\ker\phi$, therefore $\phi$ factors as $\bar{\phi}: (\wedge^2A)/D_\beta\twoheadrightarrow G\cwedge G$. Conversely, the map $G\times G\to (\wedge^2A)/D_\beta$ given by $(g,h)\mapsto (gG'\wedge hG')+D_\beta$ is a $B_0$-crossed pairing \cite{Mor12}, hence it induces a homomorphism $\psi:G\cwedge G\to (\wedge^2A)/D_\beta$ that is the inverse of $\bar{\phi}$. Hence $G\cwedge G\cong (\wedge^2A)/D_\beta$. Under this isomorphism, the map $G\cwedge G\twoheadrightarrow G'$ sends $a\cwedge b$ to $\beta(aG'\wedge bG')$, hence its kernel is isomorphic to $L_\beta/D_\beta$.
\end{proof}

Note that $\beta$ can be computed from a pc presentation of $G$. If $G=\langle x_1,\ldots ,x_d,z_1,\ldots ,z_c\rangle$ with $z_1,\ldots ,z_c$ being a basis for $G'$, then one has $[x_i,x_j]=z_1^{b_{ij1}}\cdots z_c^{b_{ijc}}$. Given a basis $e_1,\ldots ,e_d$ of $A$, we have that the matrix $B$ for $\beta$ with respect to the bases $(e_i\wedge e_j)_{1\le i<j\le d}$ and $(z_k)_{1\le k\le c}$ is $B=(b_{ijk})_{1\le i<j\le d,1\le k\le c}$. The kernel $L_\beta$ of $B$ can be computed via the Smith normal form, so the difficult part remains to generate $D_\beta$.

\begin{example}
    Let $H_{2m+1}(\mathbb Z)$ denote the $(2m+1)$-dimensional integral
    Heisenberg group.
    We show that
    $B_0\bigl(H_{2m+1}(\mathbb Z)\bigr)=0$ for all $m\ge 1$.
    Let $V\cong\mathbb Z^{2m}$ have basis
    $e_1,\ldots,e_m,f_1,\ldots,f_m$.  The commutator map
    $\beta:\wedge^2 V\rightarrow\mathbb Z$
is determined by
\begin{align*}
     \beta(e_i\wedge f_j) &=\delta_{ij},\\
     \beta(e_i\wedge e_j) &= 0\\
     \beta(f_i\wedge f_j) &=0.
\end{align*}
Its kernel $L$ is generated by
 $e_i\wedge e_j$, $f_i\wedge f_j$,
 $e_i\wedge f_j$, where $i\neq j$,
and
 $e_i\wedge f_i-e_1\wedge f_1$, where
$2\leq i\leq m$.
The generators of the first three types are decomposable elements of
$L$.  Moreover,
$
 (e_i+e_1)\wedge(f_i-f_1)\in L$
is decomposable, and
$$
 (e_i+e_1)\wedge(f_i-f_1)
 =
 e_i\wedge f_i-e_1\wedge f_1
 -e_i\wedge f_1+e_1\wedge f_i.
$$
Since the final two terms are also decomposable elements of $L$, it
follows that
$e_i\wedge f_i-e_1\wedge f_1\in D_\beta$.
Thus $L=D_\beta$, and the result is proved.
\end{example}

One can reverse the above construction as follows. Let $V\cong \Z^d$, let $W$ be a free abelian group, and let $\beta:\wedge^2V\twoheadrightarrow W$ be an epimorphism. Fix an ordered basis $e_1,\ldots ,e_d$ of $V$ and define a bilinear map $f:V\times V\rightarrow W$ by 
$$f(a,b)=\sum_{i<j}a_ib_j\beta (e_i\wedge e_j),$$
where $a=\sum_ia_ie_i$ and $b=\sum_jb_je_j$. On the set $V\times W$ define
$$(a,u)(b,v)=(a+b,u+v+f(a,b)).$$
This defines a group $G_\beta$, and it is straightforward to show that $G_\beta$ is torsion-free nilpotent of class 2 with $(G_\beta)^{\rm ab}\cong V$, $(G_\beta)'\cong W$, and $\beta$ corresponds to the commutator map on $G_\beta$. Note that $L_\beta =\ker\beta$ is a saturated lattice in $\wedge^2 V$.
Conversely, if $L$ is an arbitrary saturated lattice in $\wedge^2 V$, then $W=\wedge^2 V/L$ is torsion-free and $L$ appears as the kernel of such commutator map $\beta:\wedge^2 V\twoheadrightarrow W$.

Put $V_\Q=V\otimes \Q$ and write an element $q\in V_\Q$ as
$$q=\sum_{i<j}q_{ij}(e_i\wedge e_j).$$
For each $i<j<k<\ell$ put 
$$p_{ijk\ell}(q)=q_{ij}q_{k\ell}-q_{ik}q_{j\ell}+q_{i\ell}q_{jk}.$$
The common zero locus of these quadrics is precisely the Pl\" ucker embedding $\Gr(2,V_\Q)\hookrightarrow \mathbb{P}(\wedge^2 V_\Q)$, see \cite[p. 64]{Har92}. 
Recall that a nonzero element of a lattice is \emph{primitive} if it cannot be expressed as an integer multiple (other than $\pm 1$) of another element of the same lattice. 

\begin{lemma}
    \label{lem:intdecomp}
    Let $V$, $W$ and $\beta$ be as above, and let $L=\ker\beta$.
    \begin{enumerate}
        \item If $q\in\wedge^2 V$ is primitive and satisfies all the Pl\" ucker relations $p_{ijk\ell}(q)=0$, then $q=u\wedge v$ for some $u,v\in V$-
        \item If $P$ is a rational line in $L_\Q$, then $P\cap L$ is an infinite cyclic group. Its primitive generator is also primitive in $\wedge^2 V$.
    \end{enumerate}
\end{lemma}

\begin{proof}
    By the skew-symmetric version of the Smith normal form theorem \cite[Theorem IV.1]{New72}, there exist a basis $f_1,\ldots ,f_d$ of $V$  and positive integers $s_1\mid s_2\mid\cdots\mid s_t$ such that 
    $$q=s_1(f_1\wedge f_2)+s_2(f_3\wedge f_4)+\cdots + s_t(f_{2t-1}\wedge f_{2t}).$$
    As $q$ satisfies the Pl\" ucker relations, it is decomposable over $\Q$, so $q=\bar{u}\wedge\bar{v}$ for some $\Q$-linearly independent $\bar{u},\bar{v}\in V_\Q$. The corresponding map $q^\sharp: (V_\Q)^*\rightarrow V_Q$ maps $\lambda\in (V_\Q)^*$ to $\lambda(\bar{u})\bar{v}-\lambda(\bar{v})\bar{u}$, hence its rank is two. 
    This, in particular, shows that $t=1$, therefore $q=s_1(f_1\wedge f_2)$. But $q$ is primitive, hence $s_1=\pm 1$. This proves (1).

    It remains to prove (2). First we need to show that $P\cap L$ is nontrivial.  Let $\ell_1,\ldots ,\ell_k$ be a $\Z$-basis of $L$. Choose a nonzero $x\in P$ and write it as $x=a_1\ell_1+\cdots +a_k\ell_k$ for some $a_i\in\Q$. Then it is clear that there exists a positive integer $n$ such that $nx=m_1\ell_1+\cdots +m_k\ell_k$, where $m_i\in\Z$, hence $nx\in L\cap P$. Thus, $L\cap P$ is a nonzero subgroup of a one-dimensional rational vector space, hence it is infinite cyclic. To find its primitive generator, denote $g=\gcd(m_1,\ldots ,m_k)$ and set $q_P=(m_1/g)\ell_1+\cdots + (m_k/g)\ell_k$. Then $q_P$ is primitive in $L$. If $q_p=mw$ for some $w\in \wedge^2V$ with $|m|>1$, then $w\in P\cap \wedge^2V$. As $P\subset L_\Q$ and $L$ is saturated, $L=L_\Q\cap\wedge^2V$ and so $w\in L$, contradicting the primitivity of $q_P$ in $P\cap L$.
\end{proof}

Given a sublattice $L$ of $\wedge^2V$, set
$$X_L=\mathbb{P}(L_\Q)\cap \Gr(2,V_\Q).$$
$X_L$ is a projective variety over $\Q$ defined by the linear equations defining $L_\Q$ and the Pl\"ucker relations restricted to $\mathbb{P}(L_\Q)$. Let $X_L(\Q)$ be the set of all $\Q$-rational points of $X_L$.

\begin{proposition}
    \label{prop:XL}
    There is a bijection between $X_L(\Q)$ and the rank-one rational subspaces of $L_\Q$ generated by nonzero decomposable elements.  Furthermore, if $P\in X_L(\Q)$ and $q_P$ is the primitive integral generator of $P\cap L$, then
    $$D_\beta =\langle q_P\mid P\in X_L(\Q)\rangle_\Z.$$
    In particular, we have that $D_\beta$ is non-trivial if and only if $X_L(\Q)$ is nonempty.
\end{proposition}

\begin{proof}
    By definition, $\Q$-rational points of $\mathbb{P}(L_\Q)$ are precisely one-dimensional $\Q$-subspaces $P=\Q q$ of $L_\Q$, $q\neq 0$. Such a point belongs to $\Gr(2,V_{\mathbb Q})$ if and
    only if $q$ is decomposable over $\mathbb Q$. Thus,
    $$X_L(\mathbb Q)=\{
    \mathbb Qq\mid 
    0\neq q\in L_{\mathbb Q},\
    q=u\wedge v
    \text{ for some }u,v\in V_{\mathbb Q}\}.$$
    Let $z\in L$ be any nonzero integral decomposable element.
    Then $P=\mathbb Qz\in X_L(\mathbb Q)$.  
    Since $P\cap L=\mathbb Zq_P$, there is an integer $n\neq 0$ such that
    $z=nq_P$.
    It follows that every integral decomposable element of $L$ belongs to
    the subgroup generated by the $q_P$.  The reverse inclusion holds
    because every $q_P$ is itself an integral decomposable element of
    $L$.  Therefore,
    $D_\beta=\langle
    q_P:P\in X_L(\mathbb Q)\rangle_{\mathbb Z}$.
\end{proof}

\begin{corollary}
    \label{cor:rankcriterion}
    We have that $X_L(\Q)\neq \emptyset$ if and only if the rank of $L/D_\beta$ is strictly smaller than the rank of $L$.
\end{corollary}

\begin{proof}
    We tensor the short exact sequence of abelian groups
    $$0\rightarrow D_\beta\rightarrow L\rightarrow L/D_\beta\rightarrow 0$$
    with $\Q$. Because $\Q$ is flat over $\Z$, we get a short exact sequence of $\Q$-spaces
    $$0\rightarrow D_\beta\otimes \Q \rightarrow L\otimes \Q\rightarrow L/D_\beta \otimes \Q\rightarrow 0,$$
    and the result follows from Proposition \ref{prop:XL}.
\end{proof}

We note an algorithmic consequence of the above observations. Fix an integer $d\ge 2$. Suppose there is an unconditional terminating algorithm, which, for every integral alternating map
$\beta:\wedge^2\mathbb Z^d\longrightarrow\mathbb Z^c$
with saturated kernel $L$, computes any one of the following:
\begin{enumerate}
\item the subgroup $D_\beta$;
\item the isomorphism type of $L/D_\beta$;
\item the integer $\operatorname{rk}(L/D_\beta)$.
\end{enumerate}
Then rational-point existence is decidable for every rational linear
section $X_U=\mathbb P(U)\cap\Gr(2,d)$, where $U\leq\wedge^2\mathbb Q^d$. Namely, $L=U\cap\wedge^2\Z^d$ is saturated in $\wedge^2\Z^d$, and hence $C=\wedge^2\Z^d/L$ is free abelian. Let $\beta:\wedge^2\Z^d\twoheadrightarrow C$ be the quotient map. Then $\ker\beta =L$ and $X_U=X_L$. So each of the proposed outputs decides whether $X_U(\Q)$ is empty.

Consider the special case $d=5$, and suppose the rank of $L$ is equal to 5. Then a smooth transverse section $X_L=\mathbb{P}(L_\Q)\cap\Gr(2,5)\cong\mathbb{P}^4\cap \Gr(2,5)$ is a genus-one normal quintic \cite[Proposition 1.4]{Fis06}. Conversely, every genus-one normal quintic is a transverse linear
section of $\operatorname{Gr}(2,5)$, see \cite[Lemma 2.6]{SZ20}. No unconditional terminating algorithm is presently known which,
given an arbitrary genus-one normal quintic $C$ defined over $\Q$, 
decides whether $C(\mathbb Q)$ is nonempty.  Such a curve is a torsor under its
Jacobian $E$, and its degree-five polarization implies that its
Weil--Ch\^atelet class is either trivial or has order five.  In the
everywhere locally soluble case the class therefore belongs to
$\Sha(E/\mathbb Q)[5]$.  Deciding whether the curve has a rational
point is exactly the problem of deciding whether this class is zero.
The standard descent procedure is known to terminate under the
finiteness conjecture for Tate--Shafarevich groups; see
\cite[\S4.2]{Poo02} and
\cite[Remark~5.3]{Bak05}.  No unconditional uniform termination
theorem is currently known, including for degree-five Pfaffian models.

On a positive side, the low-dimensional cases $d\le 4$ behave better. If $d\le 3$, then every element of $\wedge^2V$ is decomposable, therefore $L_\beta=D_\beta$ and $B_0(G)=0$. Let $d=4$. Then decomposability is given by a single Pl\" ucker relation $q_{12}q_{34}-q_{13}q_{24}+q_{14}q_{23}=0$. Thus, $X_L$ is a linear section of a quadric. Rational-point existence for quadrics over $\mathbb Q$ is decidable by the Hasse–Minkowski theorem. This suggests that at least the rank of $L/D_\beta$ is effectively accessible in dimension four.

\begin{example}
Let $N_{4,2}$ be the free nilpotent group of class two on
$x_1,x_2,x_3,x_4$, and put
$$
 G=N_{4,2}/\langle [x_1,x_2][x_3,x_4]\rangle.
$$
The relator is central.  The quotient is torsion-free: on the central free
abelian group $N_{4,2}'\cong\wedge^2\Z^4$ it kills the primitive vector
$e_{12}+e_{34}$.  We have
 $G^{\rm ab}\cong\Z^4$, $G'\cong\Z^5$, and
 $L=\Z(e_{12}+e_{34})$.
$$(e_{12}+e_{34})\wedge(e_{12}+e_{34})
   =2e_{1234}\ne 0,$$
so the generator of $L$ is not decomposable.  By
Lemma \ref{lem:intdecomp}, we conclude that
  $D_\beta=0$, hence $B_0(G)\cong\Z$.
\end{example}

\section{Polynomially parametrized commuting pairs}
\label{s:poly}

\noindent
In the sequel, we need some facts on the polynomial maps between abelian groups, we refer to \cite{GT12,HK18,Lei02}. Let $B$ be an arbitrary abelian group.
If $P:\Z^d\rightarrow B$ is a map and $\bm{h}\in\Z^d$, define the difference map $\partial_{\bm{h}}P:\Z^d\rightarrow B$ by the rule $(\partial_{\bm{h}}P)(\bm{z})=P(\bm{z}+\bm{h})-P(\bm{z})$. We say that $P$ is a \emph{polynomial of degree $\le \ell$} if $\partial_{\bm{h}_{\ell+1}}\partial_{\bm{h}_{\ell}}\cdots \partial_{\bm{h}_{1}}P=0$ for all $\bm{h}_1,\ldots ,\bm{h}_{r+1}\in\Z^d$.

It is convenient to have a characterization of polynomial maps that helps immediately recognize them.
For $\bm{z}=(z_1,\ldots ,z_k)\in \Z^d$ and multi-indices $\bm{\alpha}=(\alpha_1,\ldots ,\alpha_k)\in\N^d$, set $|\bm{\alpha}| = \alpha_1+\cdots +\alpha_k$ and
    $${\bm{z}\choose \bm{\alpha}} =\prod_{i=1}^d {z_i\choose \alpha_i}.$$
Also define
    $$\partial^{\bm{\alpha}} = \partial_{e_1}^{\alpha_1}\cdots\partial_{e_d}^{\alpha_d},$$
where $e_1,\ldots ,e_d$ is the standard basis of $\Z^d$.
The following result is the Taylor expansion of polynomial maps, see, for example, Green--Tao \cite[\S8]{GT12}. We give a short proof for the sake of completeness.

\begin{proposition}
    \label{prop:newton}
    For a map $P:\Z^d\rightarrow B$, the following are equivalent.
    \begin{enumerate}
        \item $P$ is a polynomial map of degree $\le \ell$.
        \item There exist uniquely determined elements $p_{\bm{\alpha}}\in B$, where $\bm{\alpha}\in\N^d$, $|\bm{\alpha}|\le \ell$, such that
        $$P(\bm{z})=\sum_{|\bm{\alpha}|\le \ell}{\bm{z} \choose \bm{\alpha}}p_{\alpha}$$
        for all $\bm{z}\in\Z^d$. Moreover, $p_{\bm{\alpha}}=(\partial^{\bm{\alpha}}P)(\bm{0})$.
    \end{enumerate}
\end{proposition}

\begin{proof}
    First we observe that the Chu--Vandermonde identity
    $${z_i+v_i\choose \alpha_i}=\sum_{\gamma_i=0}^{\alpha_i}{v_i\choose \gamma_i}{z_i\choose \alpha_i-\gamma_i}$$
    yields
    $${\bm{z}+\bm{v}\choose \bm{\alpha}}=\sum_{\bm{\gamma}\le \bm{\alpha}}{\bm{v}\choose\bm{\gamma}}{\bm{z}\choose\bm{\alpha}-\bm{\gamma}},$$
    where the ordering $\bm{\gamma}\le \bm{\alpha}$ means $\gamma_i\le\alpha_i$ for all $i$. It follows by a straightforward calculation that
    $$\partial_{\bm{v}}\binom{\bm{z}}{\bm{\alpha}}
    =
    \sum_{\substack{\bm{\gamma}\leq\bm{}\alpha\\ \bm{\gamma}\neq 0}}
   \binom{\bm{v}}{\bm{\gamma}}
   \binom{\bm{z}}{\bm{\alpha}-\bm{\gamma}}.$$
    Thus the difference operator lowers the total degree of the multi-index binomial. This implies that the map $P$ defined in (2) is polynomial of degree $\le \ell$.

    Conversely, Let $P:\Z^d\rightarrow B$ be a polynomial map of degree $\le \ell$. For every $\bm{\alpha}$ with $|\bm{\alpha}|\le\ell$ define $p_{\bm{\alpha}}=(\partial^{\bm{\alpha}}P)(\bm{0})$ and set
    $$Q(\bm{z})=\sum_{|\bm{\alpha}|\le \ell}{\bm{z} \choose \bm{\alpha}}p_{\alpha}.$$
    One can easily show that
    $$
    \partial^{\bm{\beta}}\binom{\bm{z}}{\bm{\alpha}}
    =
    \left\{
    \begin{array}{ccc}
    \binom{\bm{z}}{\bm{\alpha}-\bm{\beta}}
       & : & \beta\leq\alpha
        \\
        0  & : & \text{otherwise}.
    \end{array}
    \right .
    $$
    This implies that $(\partial^{\bm{\beta}}Q)(0)=(\partial^{\bm{\beta}}P)(0)$ for all $|\bm{\beta}|\le\ell$. The map $R=P-Q$ is a polynomial of degree $\le\ell$ and satisfies $(\partial^{\bm{\beta}}R)(0)=0$ for all $|\bm{\beta}|\le\ell$. We prove by induction on $\ell$ that this implies $R=0$, the case $\ell=0$ being obvious. Assume the claim holds in degree $\le \ell-1$. For each $i=1,\ldots ,d$ put $S_i=\partial_{e_i}R$. The maps $S_i$ are polynomials of degree $\le\ell-1$ and, for $|\bm{\beta}|\le\ell-1$, we have $(\partial^{\bm{\beta}}S_i)(0)=(\partial^{\bm{\beta+e_i}}R)(0)=0$. The induction hypothesis implies $S_i=0$, hence $R(\bm{z}+e_i)=R(\bm{z})$ for all $i=1,\ldots ,d$ and all $\bm{z}\in\Z^d$. This quickly implies $R=0$. 
    Uniqueness of $p_{\bm{\alpha}}$ is straightforward.
\end{proof}

The following is an immediate corollary of \cite[Proposition 1.15]{Lei02}, it also follows from Proposition \ref{prop:newton}:

\begin{proposition}
    \label{prop:fingen}
    Let $P:\Z^d\rightarrow B$ be a polynomial map of degree $\ell$. Then
    $$\langle P(\bm{z})\mid \bm{z}\in\Z^d\rangle =\langle P(\bm{\alpha})\mid \bm{\alpha}\in\N^d, |\bm{\alpha}|\le \ell\rangle.$$
\end{proposition}

Suppose  $G$ has a normal torsion-free subgroup $H$ of finite index. This is true, for example, for
polycyclic-by-finite groups \cite[Chapter 1, Proposition 2]{Seg05}. Furthermore, if $G$ is virtually nilpotent (abelian), then $H$ can be chosen to be nilpotent (abelian). Write $Q=G/H$ and choose a section $s:Q\rightarrow G$, $q\mapsto s_q$, with $s_1=1$. Given $q,r\in Q$, put
$$\mathcal{C}_{q,r}=\{(hs_q,ks_r)\in Hs_q\times Hs_r\mid [hs_q,ks_r]=1\}.$$
Note that the set $\mathcal{C}_{q,r}$ may be empty. 

\begin{definition}
    A \emph{polynomial cover of the commuting loci} of $G$ consists, for each $q,r\in Q$ with $\mathcal{C}_{q,r}\neq\emptyset$, of a list of maps $$\phi_{q,r,j}:\Z^{d_{q,r,j}}\rightarrow \mathcal{C}_{q,r},$$ and nonnegative integers $\ell_{q,r,j}$, $j=1,\ldots ,t_{q,r}$, such that the union of the images of $\phi_{q,r,j}$ equals $\mathcal{C}_{q,r}$, and the maps
    $$\theta_{q,r,j}:\Z^{d_{q,r,j}}\rightarrow M(G)$$
    given by $\theta_{q,r,j}(\bm{z})=\operatorname{pr}_1(\phi_{q,r,j}(\bm{z}))\wedge \operatorname{pr}_2(\phi_{q,r,j}(\bm{z}))$ are polynomial of degree $\le\ell_{q,r,j}$.
\end{definition}

Suppose $G$ admits a polynomial cover of its commuting loci. Let $x,y\in G$ commute. Write $x=hs_q$ and $y=ks_r$, so $(x,y)\in\mathcal{C}_{q,r}$. There exists $j$ such that $(x,y)=\phi_{q,r,j}(\bm{z}_0)$ for some $\bm{z}_0\in\Z^{d_{q,r,j}}$. By Proposition \ref{prop:fingen}, we have that $$x\wedge y=\theta_{q,r,j}(\bm{z}_0)\in \langle \theta_{q,r,j}(\bm{\alpha})\mid |\bm{\alpha}|\le \ell_{q,r,j}\rangle .$$
This implies that $M_0(G)$ is generated by all $\theta_{q,r,j}(\bm{\alpha})$, where $q,r,\in Q$, $j=1,\ldots ,t_{q,r}$, $|\bm{\alpha}|\le \ell_{q,r,j}$. As $Q$ is finite, this set of generators is finite. In particular, the Bogomolov multiplier of a polycyclic group admitting a polynomial cover of its commuting loci can be effectively computed.

From here on we restrict to the case when $G$ is finitely generated nilpotent-by-finite. Without loss of generality, $H$ is torsion-free nilpotent of class $c$. Fix a Mal'cev basis of $\mathcal{H}=(h_1,\ldots ,h_d)$ of $H$. We have a bijection $\mu_{\mathcal{H}}:\Z^d\rightarrow H$ given by $\mu_{\mathcal{H}}(\bm{z})=h_1^{z_1}\cdots h_d^{z_d}$.

A map $a:\phi:\Z^d\rightarrow \Z^n$ is said to have \emph{coordinate degree at most $\ell$} if each coordinate is an integer-valued polynomial of total degree $\le \ell$.

\begin{proposition}
    \label{prop:wedgefromcoord}
    Let $G$ be finitely generated nilpotent-by-finite and $H$ a torsion-free nilpotent normal subgroup of $G$ of finite index. Let $c$ be the nilpotency class of $H$.
    Fix $q,r\in G/H$ and suppose that $\mathcal{C}_{q,r}\neq\emptyset$. There exists an effectively computable integer $E_{q,r}\le c+1$ with the following property. Suppose we have a map $\phi:\Z^e\rightarrow\mathcal{C}_{q,r}$ given by $\phi(\bm{z})=(\mu_{\mathcal{H}}(a(\bm{z}))s_q,\mu_{\mathcal{H}}(b(\bm{z}))s_r)$, where $a,b:\Z^e\rightarrow \Z^d$ have coordinate degree $\le \ell$. Then the map $\theta:\Z^e\rightarrow M(G)$ given by
    $$\theta(\bm{z})=\mu_{\mathcal{H}}(a(\bm{z}))s_q\wedge \mu_{\mathcal{H}}(b(\bm{z}))s_r$$
    is a polynomial map of degree $\le \ell E_{q,r}$.
\end{proposition}

\begin{proof}
    As $G$ is finitely presented, choose a finite free presentation $G=F/R$. Put $\widehat{G}=F/[F,R]$ and $N=R/[F,R]$. Let $\widehat{H}$ be inverse image of $H$ in $\widehat{G}$. As $N$ is central in $\widehat{G}$, it is finitely generated. Thus it follows that $\widehat{H}$ is finitely generated nilpotent of class $\le c+1$, where $c$ is the nilpotency class of $H$. For each $i$, choose a lift $\widehat h_i\in\widehat H$ of $h_i$, and
    define a set-theoretic lift of the Mal'cev coordinate map by
    $\widehat\mu_{\mathcal H}(\bm{z})=\widehat h_1^{z_1}\cdots\widehat h_d^{z_d}$.
    Also choose fixed lifts $t_q,t_r\in\widehat G$ of $s_q,s_r$. As $\mathcal{C}_{q,r}\neq\emptyset$, it follows, in particular, that $[q,r]=1$. This implies that  $$[\widehat\mu_{\mathcal H}(\bm{u})t_q,\widehat\mu_{\mathcal H}(\bm{v})t_r]
    \in\widehat H$$ for all $\bm{u},\bm{v}\in\Z^d$.
    Choose a consistent nilpotent generating sequence for $\widehat H$
    adapted to the central subgroup $N$, and choose generators
    $n_1,\ldots,n_s$ of $N$. We apply Hall collection ordered by commutator weight to the preceding commutator. In a nilpotent group of class $k$, the exponent of a collected commutator of weight $w$ is an integer-valued polynomial of ordinary total degree at most $w$ in the input exponents.  All commutators of weight greater than $k$ vanish. 
    This
    shows that there exist integer-valued rational polynomials $P_i:\Z^{2d}\rightarrow\Z$, and $\gamma_{q,r}(\bm{u},\bm{v})\in\Z^d$ with coordinates that are integer-valued rational polynomials, such that
    $$[\widehat\mu_{\mathcal H}(\bm{u})t_q,\widehat\mu_{\mathcal H}(\bm{v})t_r]=\widehat\mu_{\mathcal H}(\gamma_{q,r}(\bm{u},\bm{v}))
    n_1^{P_1(\bm{u},\bm{v})}\cdots n_s^{P_s(\bm{u},\bm{v})}.$$
    Let $$E_{q,r} =\max\{\deg P_j\mid 1\leq j\leq s\}\le c+1.$$ Let $\bm{z}\in\Z^e$.
    As $(\mu_{\mathcal H}(a(\bm{z}))s_q,\mu_{\mathcal H}(b(\bm{z}))s_r)\in\mathcal C_{q,r}$,
    we conclude that 
    $\mu_{\mathcal H}(\gamma_{q,r}(a(\bm{z}),b(\bm{z})))=1$.
    Uniqueness of Mal'cev coordinates in the torsion-free group $H$ implies
    $\gamma_{q,r}(a(\bm{z}),b(\bm{z}))=0$. Therefore,
    $$\theta(\bm{z})=[\widehat\mu_{\mathcal H}(a(\bm{z}))t_q,\widehat\mu_{\mathcal H}(b(\bm{z}))t_r]=n_1^{P_1(a(\bm{z}),b(\bm{z}))}\cdots n_s^{P_s(a(\bm{z}),b(\bm{z}))}.$$
    Each coordinate of $a(\bm{z})$ and $b(\bm{z})$ has total degree at most $\ell$.
    Therefore,
    $\deg P_j(a(\bm{z}),b(\bm{z}))
    \leq \ell\deg P_j
    \leq \ell E_{q,r}$.
    In additive notation,
    $$\theta(\bm{z})=\sum_{j=1}^sP_j(a(\bm{z}),b(\bm{z}))\,n_j.$$
    It is therefore an $M(G)$-valued polynomial map of degree at most
    $\ell E_{q,r}$.
\end{proof}


\section{Virtually abelian groups}
\label{s:virtabel}

\noindent
At first, we prove the following result:

\begin{proposition}
    \label{prop:fitting}
    Let $G$ be a finitely generated virtually abelian group. Let $F$ be its Fitting subgroup. Then the following hold:
    \begin{enumerate}
        \item $|F:Z(F)|$ is finite.
        \item Let $m$ be the exponent of the torsion group of $Z(F)$. Then $A=mZ(F)$ is a torsion-free normal abelian subgroup of $G$ with $|G:A|<\infty$.
    \end{enumerate}
\end{proposition}

\begin{proof}

    Let $H$ be an abelian normal subgroup of $G$ of finite index. Then $H\le F$, hence $|G:F|<\infty$. Thus it suffices to show that $|F:Z(F)|<\infty$. Let $B=eH$, where $e$ is the exponent of the torsion subgroup of $H$. Then $B$ is characteristic in $H$ and hence normal in $G$. Note that $|F:B|<\infty$. Conjugation by $x\in F$ induces an automorphism  $\alpha_x$ of $B$. As $x^m\in B$ for some $m\ge 1$, it follows that $\alpha_x^m=1$. As $G$ is finitely generated virtually abelian, it satisfies the max condition, hence $F$ is nilpotent. Let $c$ be its nilpotency class, and set
    $V_i=\bigl(B\cap Z_i(F)\bigr)\otimes_{\mathbb Z}\mathbb Q\leq B_{\mathbb Q}:=B\otimes_{\mathbb Z}\mathbb Q$.
    Then
    $$0=V_0\leq V_1\leq\cdots\leq V_c=B_{\mathbb Q}.$$
    For $b\in B\cap Z_i(F)$ and $x\in F$, we have that $[b,x]\in B\cap Z_{i-1}(F)$.
    Consequently, conjugation by $x$ acts trivially on each quotient
    $V_i/V_{i-1}$.  Thus the automorphism $\alpha_x$ induced by conjugation
    on $B_{\mathbb Q}$ is unipotent. A finite-order unipotent linear transformation over a field of characteristic zero is the identity.  Therefore $\alpha_x=1$, so $x$
    centralizes $B$.  Since $x\in F$ was arbitrary, $B\leq Z(F)$.
    This proves (1), and (2) is a straightforward consequence.
\end{proof}

Let $G$ be a finitely generated virtually abelian group. Let $A\cong \Z^n$ be a normal subgroup (written additively) of $G$  with $|G:A|<\infty$. Write $Q=G/A$. Choose a section $s:Q\rightarrow G$, $q\mapsto s_q$, with $s_1=1$. Then $A$ becomes a $Q$-module via $q\cdot a=s_qas_q^{-1}$, and this yields a map $\rho:Q\rightarrow\GL_n(\Z)$. The map $\omega:Q\times Q\rightarrow A$ given by $\omega(q,r)=s_qs_rs_{qr}^{-1}$ is the factor set of the chosen section. The group $G$ can be identified, as a set, with $A\times Q$, together with the multiplication
$$(a,q)(b,r)=(a+\rho(q)b+\omega(q,r),qr).$$
This gives a description of commuting pairs in $G$. The elements $(a,q)$ and $(b,r)$ commute if and only if
\begin{align}
    \label{eq:comm1} qr &= rq,\\ 
    \label{eq:comm2} (I-\rho(r))a+(\rho(q)-I)b &= \omega(r,q)-\omega(q,r).
\end{align}
Fix a basis of $A$ and use row vectors. Let $(a,q)$ and $(b,r)$ commute. We define a $2n\times n$ integer matrix
$$T_{q,r}=\begin{pmatrix}
    I - \rho(r)\\
    \rho(q) - I
\end{pmatrix}.$$
Denote also $\Delta_{q,r}=\omega(r,q)-\omega(q,r)$.
Then the equation \eqref{eq:comm2} can be rewritten as $(a,b)T_{q,r}=\Delta_{q,r}$.

Let $q,r\in Q$, and let $\mathcal{C}_{q,r}$ be the set commuting pairs in $As_q\times As_r$. If $q$ and $r$ do not commute, then $\mathcal{C}_{q,r}$ is empty. Suppose that $qr=rq$. Then one can use the Smith normal form to decide whether the system $xT_{q,r}=\Delta_{q,r}$ has integral solutions. If not, then $\mathcal{C}_{q,r}=\emptyset$. If there is some integral solution $\bm{x}_0\in\Z^{2n}$, then the set of all solutions is $\bm{x}_0+\Lambda_{q,r}$, where $\Lambda_{q,r}$ is the (left) integral kernel of $T_{q,r}$. Again, Smith normal form can be used to determine both $\bm{x}_0$ and an integral basis $k_1,\ldots ,k_d$ of $\Lambda_{q,r}$; note that $d$ depends on $q$ and $r$, so we sometimes write $d=d_{q,r}$. Using the notation $\bm{z}=(z_1,\ldots ,z_d)\in\Z^d$, we see that
$$\mathcal{C}_{q,r}=\{ (a(\bm{z})s_q,b(\bm{z})s_r)\mid \bm{z}\in\Z^d, (a(\bm{z}),b(\bm{z}))=\bm{x}_0+z_1k_1+\cdots +z_dk_d\}.$$
If $\mathcal{C}_{q,r}\neq\emptyset$, we
have a map $$\phi_{q,r}:\mathbb{Z}^{d_{q,r}}\longrightarrow \mathcal{C}_{q,r}$$ given by $\phi_{q,r}(\bm{z})=(a(\bm{z})s_q,b(\bm{z})s_r)$. Note that the coordinate functions of $a$ and $b$ are affine and therefore of degree $\le 1$. Proposition \ref{prop:wedgefromcoord} implies that $G$ admits a polynomial cover of commuting loci of degree $\le 2$. Note that if $P:\Z^d\rightarrow B$ is a polynomial map of degree $\le 2$, then Proposition \ref{prop:fingen} implies that $\langle P(\bm{z})\mid \bm{z}\in\Z^d\rangle =\langle P(0),P(e_i),P(2e_i),P(e_i+e_j)\mid 1\le i,j\le d, i<j\rangle$.

The algorithm for computing $B_0(G)$ for a polycyclic group $G$ known to be virtually abelian therefore goes as follows:

\begin{enumerate}
    \item Construct $A\cong\Z^n$ as in Proposition \ref{prop:fitting}. Set $Q=G/A$ and choose a transversal $s:Q\rightarrow G$.
    \item Compute the action matrices $\rho(q)$ and factor-set vectors $\omega(q,r)$.
    \item Compute $G\wedge G$, the epimorphism $\kappa:G\wedge G\twoheadrightarrow G'$, its kernel $M$ and the crossed pairing $\lambda:G\times G\rightarrow G\wedge G$.
    \item For every commuting $(q,r)\in Q\times Q$, solve the system $$(I-\rho(r))\bm{a}+(\rho(q)-I)\bm{b}=\omega(q,r)-\omega(r,q)$$ for $\bm{a},\bm{b}\in\Z^n$. The solutions are parametrized by
    $$(a(\bm{z}),b(\bm{z}))=\bm{x}_0+z_1k_1+\cdots +z_dk_d,$$
    see above. This gives a description of the set $$\mathcal{C}_{q,r}=\{(a(\bm{z})s_q,b(\bm{z})s_r)\mid \bm{z}\in\Z^d\}.$$
    \item Evaluate $\lambda$ on the above pairs for $\bm{z}\in\{ 0,e_i,2e_i, e_i+e_j\mid 1\le i,j\le d, i<j\}$
    and append the results to a set $S$.
    \item Return $M/\langle S\rangle$.
\end{enumerate}



\section{Bogomolov multipliers and extensions of groups}
\label{s:B0metabelian}

\noindent

\subsection{Abelian-by-cyclic groups}
\label{ss:abbycyc}

Bogomolov \cite[Lemma 4.9]{Bog88} showed that if $G$ is a finite abelian-by-cyclic group, then its (cohomological) Bogomolov multiplier is trivial. Another proof of this fact was found by Jezernik \cite{Jez16} in his PhD thesis. Both proofs essentially use the fact that the group in question is finite. We show here that we can drop the finiteness assumption:

\begin{theorem}
    \label{thm:abbycyc}
    If $G$ is an abelian-by-cyclic group, then $B_0(G)=0$.
\end{theorem}

\begin{proof}
    Let $A$ be an abelian normal subgroup of $G$ such that
    $C=G/A$ is cyclic.  Consider the homological
    Lyndon--Hochschild--Serre spectral sequence
    \cite[Chapter~VII, Section~6]{Bro82}
    $$
    E^2_{p,q}=H_p(C,H_q(A,\mathbb Z))
    \Longrightarrow H_{p+q}(G,\mathbb Z).
    $$
    In total degree two this gives a filtration
    $0=F_{-1}\subseteq F_0\subseteq F_1\subseteq F_2
   =H_2(G,\mathbb Z)$
    such that
    $F_p/F_{p-1}\cong E^\infty_{p,2-p}$.
    Since $C$ is cyclic, the group $H_2(C,\mathbb Z)$ is trivial.    
    Thus $E^2_{2,0}=0$, and consequently
    $E^\infty_{2,0}=0$ and $H_2(G,\Z)=F_2=F_1$.
    The subgroup $F_0$ is the image of the edge homomorphism
    $H_0(C,H_2(A,\mathbb Z))
    =H_2(A,\mathbb Z)_C
    \rightarrow H_2(G,\mathbb Z)$.
    Since $A$ is abelian,
    $H_2(A,\mathbb Z)\cong\bigwedge\nolimits^2_{\mathbb Z}A$,
    and this group is generated by elements $a\wedge b$ with
    $a,b\in A$.  Their images in $H_2(G,\mathbb Z)$ are the torus
    classes $\tau_G(a,b)$.  Hence
    $F_0\leq M_0(G)$.

    Consider $F_1/F_0\cong E^\infty_{1,1}$.
    There are no outgoing differentials from bidegree $(1,1)$.
    The only possible incoming differential is
    $d^2\colon E^2_{3,0}\longrightarrow E^2_{1,1}$.
    Therefore
    $E^\infty_{1,1}=E^2_{1,1}/\im d^2=H_1(C,A)/\im d^2$.

We now recall the description of $H_1(C,A)$.  If
$C=\langle c\rangle\cong\mathbb Z$, then
$H_1(C,A)=A^C$.
If $C=\langle c\mid c^r=1\rangle$ is finite cyclic, then
 $H_1(C,A)=A^C/N_C(A)$,
 where  $N_C=1+c+\cdots+c^{r-1}$ is the norm map for the action of $C$ on $A$.
Thus, in either case, every element of $H_1(C,A)$ is represented
by an element $a\in A^C$.

Choose a lift $t\in G$ of the generator $c\in C$.  The
$C$-action on $A$ is induced by conjugation, so
$c\cdot a=tat^{-1}$.
Since $a\in A^C$, we have $tat^{-1}=a$, and therefore $t$ and $a$ commute.
Consequently, the pair $(t,a)$ defines a torus class
$\tau_G(t,a)\in M_0(G)$.
Under the LHS filtration, the class $\tau_G(t,a)$ belongs to
$F_1$, and its image under
$F_1\longrightarrow F_1/F_0
 \cong E^\infty_{1,1}$
is the image of the class represented by $a\in H_1(C,A)$.
Indeed, in the normalized bar resolution the torus is represented
by $[t\mid a]-[a\mid t]$,
whose filtration-one component is the standard cycle representing
$a$ in $H_1(C,A)$.

It follows that
$M_0(G)\cap F_1\longrightarrow F_1/F_0$
is surjective.  Since $F_0\leq M_0(G)$, this implies
$F_1\leq M_0(G)$.
Namely, if $x\in F_1$, choose $m\in M_0(G)\cap F_1$ having the
same image as $x$ in $F_1/F_0$.  Then
$x-m\in F_0\leq M_0(G)$,
so $x\in M_0(G)$. This proves the result.
\end{proof}

\subsection{Semidirect products}
\label{ss:semidirect}
A formula for Bogomolov multipliers of semidirect products $G=N\rtimes Q$ of groups $N$ and $Q$ is obtained in \cite[Section 6]{Mor12}. Note that the standing assumption there is that the groups are finite, but the homological description, in particular, \cite[Theorem 6.1]{Mor12}, still works for arbitrary groups. In the case when $G$ is finite and $(|N|,|Q|)=1$, a cleaner formula was obtained by Kang \cite{Kan14}.

Here we obtain an alternative description for $B_0(G)$ that works for arbitrary groups. Let $\pi:G\rightarrow Q$ be the projection and fix a homomorphic section $s:Q\rightarrow G$. Put
$$K(G,N)=\ker (H_2(\pi):H_2(G,\Z)\to H_2(Q,\Z)).$$
Let $x,y\in G$ commute, and set
$$\delta_s(x,y)=\tau_G(x,y)-H_2(s)\tau_Q(\pi(x),\pi(y)).$$
Note that $\delta_s(x,y)\in K(G,N)$. Define
$$D_s(G,N)=\langle \delta_s(x,y)\mid x,y\in G, [x,y]=1\rangle.$$

\begin{theorem}
\label{thm:B0split}
    If 
    $$
    \begin{tikzcd}
    1 \arrow[r] &
    N \arrow[r] &
    G \arrow[r, shift left=.4ex, "\pi"] &
    Q  \arrow[l, shift left=.4ex, "s"]\arrow[r] &
    1
\end{tikzcd}
$$
is a split exact sequence of groups, there is a split exact sequence
    $$
    \begin{tikzcd}
    0 \arrow[r] &
    \frac{K(G,N)}{D_s(G,N)} \arrow[r] &
    B_0(G) \arrow[r, , "B_0(\pi)"] &
    B_0(Q) \arrow[r] &
    0.
\end{tikzcd}
$$
The splitting is induced by $s$. 
\end{theorem}

\begin{proof}
    As $\pi s=id_Q$, Functoriality gives $H_2(\pi)H_2(s)=id_{H_2(Q,\Z)}$. Thus $H_2(G,\Z)=K(G,N)\oplus H_2(s)H_2(Q,\Z)$. If $x,y\in G$ commute, the equation
    $\tau_G(x,y)=\delta_s(x,y)+H_2(s)\tau_Q(\pi(x),\pi(y))$ gives $M_0(G)=D_s(G,N)+H_2(s)M_0(Q)$. This sum is clearly direct, as $D_s(G,N)\le K(G,N)$. This finishes the proof.
\end{proof}

We assume from here on that $G=A\rtimes Q$, where $A$ is abelian. In this case, $A$ can be considered as a $G$-module. The following can be seen as a homological counterpart of \cite[Theorem 2]{Tah72}, proved there for finite groups:

\begin{lemma}
\label{lem:spect}
    There is a natural exact sequence
        $$
    \begin{tikzcd}
    H_2(A,\Z)_Q \arrow[r] &
    K(G,A) \arrow[r, , "\epsilon"] &
    H_1(Q,A) \arrow[r] &
    0.
\end{tikzcd}
$$
The image of the first map is contained in $M_0(G)$.
\end{lemma}

\begin{proof}
    Consider the homological Lyndon--Hochschild--Serre spectral sequence
    $$E^2_{p,q}=H_p(Q,H_q(A,\mathbb Z))
    \Longrightarrow H_{p+q}(G,\mathbb Z).$$
    Its total-degree-two filtration is
    $0\subseteq F_0\subseteq F_1\subseteq F_2=H_2(G,\mathbb Z)$,
    where $F_p/F_{p-1}\cong E^\infty_{p,2-p}$.

    Let $\pi: G\rightarrow Q$ be the projection and
    $s: Q\rightarrow G$ a section.  Since $\pi_*s_*=id$, the horizontal
    edge homomorphism $\pi_*\colon H_p(G,\mathbb Z)\rightarrow H_p(Q,\mathbb Z)$
    is surjective.  It follows that the bottom row survives unchanged, that is,
    $$E^\infty_{p,0}=E^2_{p,0}=H_p(Q,\mathbb Z).$$
    This is the homological counterpart of
     \cite[Proposition 7.3.2]{Eve91}. Consequently,
    $$F_1=\ker\bigl(H_2(G,\mathbb Z)\to H_2(Q,\mathbb Z)\bigr)=K(G,A).$$
    Because $A$ is abelian,
    $E^2_{1,1}=H_1(Q,H_1(A,\mathbb Z))=H_1(Q,A)$.   
    The only possible nonzero differential involving this term is
    $d^2_{3,0}\colon E^2_{3,0}\longrightarrow E^2_{1,1}$,
    but it vanishes because the bottom row survives unchanged.  Hence
    $$E^\infty_{1,1}=H_1(Q,A).$$
    The filtration therefore gives an exact sequence
    $$F_0\longrightarrow K(G,A)\longrightarrow H_1(Q,A)\longrightarrow0.$$
    The vertical edge homomorphism
    $H_2(A,\mathbb Z)_Q\rightarrow H_2(G,\mathbb Z)$    
    has image $F_0$.  This gives the exact sequence
    $$H_2(A,\mathbb Z)_Q\longrightarrow K(G,A)\longrightarrow H_1(Q,A)\longrightarrow0.$$
    Finally,
    $H_2(A,\mathbb Z)\cong\bigwedge\nolimits^2 A$,
    and $a\wedge b$ maps to  $\tau_G(a,b)$.
    Therefore the image of $H_2(A,\mathbb Z)_Q$ is contained in $M_0(G)$.
 \end{proof}

 \begin{proposition}
    \label{prop:B0AQ}
    Let
    $$
    \begin{tikzcd}
    1 \arrow[r] &
    A \arrow[r] &
    G \arrow[r, shift left=.4ex, "\pi"] &
    Q  \arrow[l, shift left=.4ex, "s"]\arrow[r] &
    1
\end{tikzcd}
$$
be a split exact sequence of groups with $A$ abelian. For every commuting pair $x,y\in G$ set
$$t_s(x,y)=\epsilon(\delta_s(x,y)),$$
where $\epsilon: K(G,A)\twoheadrightarrow H_1(Q,A)$ is as in Lemma \ref{lem:spect}. Define
$$T(A,Q)=\langle t_s(x,y)\mid x,y\in G,[x,y]=1\rangle.$$
Then $B_0(G)\cong B_0(Q)\oplus H_1(Q,A)/T(A,Q)$.
 \end{proposition}

 \begin{proof}
    By Theorem \ref{thm:B0split}, $B_0(G)\cong B_0(Q)\oplus K(G,A)/D_s(G,A)$. Let $F_0$ be the image of $H_2(A,\Z)_Q\to K(G,A)$. Then Lemma \ref{lem:spect} shows that $F_0\le K(G,A)\cap M_0(G)=D_s(G,A)$. Hence $K(G,A)/D_s(G,A)\cong (K(G,A)/F_0)/(D_s(G,A)/F_0)\cong H_1(Q,A)/T(A,Q)$, as required.
 \end{proof}

  Barge \cite[Theorem 3]{Bar89} proved that if $A$ and $Q$ are
finite, $A$ is abelian, and $Q$ is bicyclic, then
$B_0(A\rtimes Q)=0$.  We now show that the corresponding statement
for the homological Bogomolov multiplier holds without any finiteness
assumption.

\begin{lemma}
\label{lem:torus-edge}
Let $G=A\rtimes Q$, where $A$ is abelian, and let
$s(q)=(0,q)$ be the canonical section. Suppose that
$x,y\in Q$ commute and that $u,v\in A$ satisfy
$$
   (1-y)u+(x-1)v=0.
$$
Then $g=(u,x)$ and $h=(v,y)$ commute in $G$. Moreover, under
the epimorphism
$\epsilon\colon K(G,A)\longrightarrow H_1(Q,A)$
of Lemma \ref{lem:spect}, one has
$$
   \epsilon\bigl(\delta_s(g,h)\bigr)
     =[\,v e_x-u e_y\,].
$$
\end{lemma}

\begin{proof}
The equalities
 $(u,x)(v,y)=(u+xv,xy)$ and 
$(v,y)(u,x)=(v+yu,yx)$ show
that $g$ and $h$ commute precisely when
$(1-y)u+(x-1)v=0$.

Let $\Lambda=\langle X,Y\mid [X,Y]=1\rangle\cong\mathbb Z^2$
and let $\varphi\colon\Lambda\to Q$ be given by
$\varphi(X)=x$ and $\varphi(Y)=y$. Then $A$ becomes a 
$\Lambda$-module. Put
$\widetilde G=A\rtimes_\varphi\Lambda$.

As  $(1-y)u+(x-1)v=0$, 
the assignments
$\eta(X)=u$ and $\eta(Y)=v$
extend to a crossed homomorphism $\eta\colon\Lambda\to A$:
Consequently,
$\sigma_\eta(\lambda)=(\eta(\lambda),\lambda)$
and
$\sigma_0(\lambda)=(0,\lambda)$
define two splittings of \(\widetilde G\to\Lambda\).

We use Wall's filtered resolution for this extension
\cite{Wal61}.  Its associated graded complex is the tensor-product
complex of resolutions for \(\Lambda\) and \(A\), and its filtration
by the \(\Lambda\)-degree gives the LHS spectral sequence.  In total
degree two the associated graded complex has components of bidegrees
$(2,0)$ $(1,1)$ and $(0,2)$.
Let
$$
   \Delta_\eta
   =(\sigma_\eta)_*[T^2]-(\sigma_0)_*[T^2].
$$
The two terms have the same degree-$(2,0)$ component, namely the
fundamental class of the base torus, so this component cancels.
The two sections differ along the $X$-edge by $u$ and along the
$Y$-edge by $v$.  With the orientation $X\wedge Y$, the
degree-$(1,1)$ component of their difference is therefore
$e_X\otimes v-e_Y\otimes u$,
which is identified with
$v e_X-u e_Y$
in the degree-one complex computing $H_1(\Lambda,A)$.

Any remaining degree-$(0,2)$ component belongs to
$$F_0=
 \im\bigl(
 H_2(A,\mathbb Z)_\Lambda
 \longrightarrow H_2(\widetilde G,\mathbb Z)
 \bigr).
$$
Consequently, under the isomorphism
$F_1/F_0\cong H_1(\Lambda,A)$,
we have
$$
   \epsilon_\Lambda(\Delta_\eta)
      =[\,v e_X-u e_Y\,].
$$
The homomorphism
$\Phi\colon\widetilde G\longrightarrow G$ given by
$(a,\lambda)\longmapsto(a,\varphi(\lambda))$
takes $\Delta_\eta$ to $\delta_s(g,h)$.  Naturality of the LHS
filtration therefore gives
$$
 \epsilon\bigl(\delta_s(g,h)\bigr)
 =
\bigl[\,v e_X-u e_Y\,\bigr].
$$
This concludes the proof.
\end{proof}

\begin{theorem}
\label{thm:abbybicyc}
Let $A$ be an abelian group and let $Q$ be bicyclic.  Then, for
every action of $Q$ on $A$,
$B_0(A\rtimes Q)=0$.
\end{theorem}

\begin{proof}
    If $Q$ is cyclic, the assertion follows from
Theorem~ \ref{thm:abbycyc}. It remains to consider
$Q=\langle x\rangle\times\langle y\rangle
     \cong C_r\times C_s$,
where either factor may be finite or infinite. By
Proposition \ref{prop:B0AQ}, it suffices to prove that
$H_1(Q,A)=T(A,Q)$.

Using the tensor product of the standard resolutions of the two cyclic
factors, the degree-one differential is
$d_1: Ae_x\oplus Ae_y\longrightarrow A$ given by
$d_1(ae_x+be_y)=(x-1)a+(y-1)b$.
Thus every class in $H_1(Q,A)$ has a representative
$ae_x+be_y$ satisfying the cycle condition
$(x-1)a+(y-1)b=0$.

Set $g=(-b,x)$ and $h=(a,y)$.
The preceding cycle condition is
$(1-y)(-b)+(x-1)a=0$,
so Lemma~\ref{lem:torus-edge} applies with $u=-b$ and $v=a$.
It gives
$$t_s(g,h)
   =\epsilon\bigl(\delta_s(g,h)\bigr)
   =[\,ae_x+be_y\,].$$
Hence every element of $H_1(Q,A)$ belongs to $T(A,Q)$, and
therefore
$H_1(Q,A)=T(A,Q)$.
\end{proof}

\begin{example}
    \label{ex:rank3}
    Let $U$ and $V$ be free abelian of rank three with bases $e_1,e_2,e_3$ and $v_1,v_2,v_3$, respectively. Put
    $$R=\langle e_1\otimes v_1+e_2\otimes v_2+e_3\otimes v_3\rangle\le U\otimes V.$$
    Let $W=(U\otimes V)/R\cong \Z^8$. Define $\alpha:U\otimes V\rightarrow W$ to be the quotient map. Let $A=V\oplus W$. For $u\in U$, define an automorphism $T_u$ of $A$ by the rule
    $$T_u(v,w)=(v,w+\alpha(u\otimes v)).$$
    The maps $T_u$ commute, and $(T_u-I)(T_{u'}-I)=0$. Hence $G=A\rtimes U$ is torsion-free nilpotent of class 2 with $G^{\rm ab}\cong U\oplus V$ and $G'=W$. Let $\beta:\wedge^2(U\oplus V)\rightarrow W$ be the commutator map. Note that $\beta$ is zero on both $\wedge^2 U$ and $\wedge^2V$, whereas $\beta$ agrees with $\alpha$ on the crossed term $U\otimes V$. Hence
    $$L_\beta=\ker\beta =\wedge^2U\oplus R\oplus \wedge^2V.$$
    The terms $\wedge^2U$ and $\wedge^2V$ are generated by decomposable elements. Consider an arbitrary decomposable element
    $$(u+v)\wedge (u'+v')=u\wedge u'+u\otimes v'-u'\otimes v+v\wedge v'$$
    in $L_\beta$. Its cross component $u\otimes v'-u'\otimes v$, viewed as a $3\times 3$ matrix, has rank $\le 2$. On the other hand, every nonzero element of $R$ is a scalar matrix and thus has rank 3 over $\Q$. Thus no nonzero element of $R$ is generated by decomposable elements of $L_\beta$. Consequently, $B_0(G)\cong L_\beta/D_\beta\cong R\cong\Z$. 
\end{example}

\subsection{Wreath products}
\label{ss:wreath}
Let $H$ and $Q$ be groups. Consider their restricted regular wreath product $W=H\wr Q$. Let 
$$N=H^{(Q)}=\bigoplus_{q\in Q}H$$
be its base group. Then $W=N\rtimes Q$. Let $\pi:W\to Q$ be the projection, an $s:Q\to W$ a homomorphic section. Let $\iota:H\hookrightarrow W$ be the embedding of $H$ in the coordinate indexed by $1\in Q$. We have that
$$N^{\rm ab}=\bigoplus H^{\rm ab}$$
is isomorphic to $\Z Q\otimes_{\Z} H^{\rm ab}$ as a $Q$-module. Hence Shapiro's lemma gives $H_i(Q,N^{\rm ab})=0$ for all $i>0$. Consider the LHS spectral sequence
$$E_{p,q}^2=H_p(Q,H_q(N,\Z))\Rightarrow H_{p+q}(W,\Z).$$
By the above, $E_{1,1}^2=H_1(Q,N^{\rm ab})=0$ and $E_{2,1}^2=H_2(Q,N^{\rm ab})=0$. The only possible differentials entering the term in bidegree $(0,2)$ are $d_{2,1}^2:E_{2,1}^2\rightarrow E_{0,2}^2$ and $d_{3,0}^3:E_{3,0}^3\rightarrow E_{0,2}^3$. The first one has zero source, the second one is zero since we have a split extension, the argument being similar to that of the proof of Lemma \ref{lem:spect}. Therefore $E_{0,2}^\infty=E_{0,2}^2=H_0(Q,H_2(N,Z))=H_2(N,\Z)_Q$. In addition to that, $E_{1,1}^\infty=0$ and $E_{2,0}^\infty=H_2(Q,\Z)$. The total-degree-two filtration yields that the inclusion $N\hookrightarrow W$ induces an isomorphism
$$H_2(N,\Z)_Q\cong \ker (H_2(W,Z)\rightarrow H_2(Q,\Z)).$$

\begin{theorem}
    \label{thm:wreath}
    For arbitrary groups $H$ and $Q$, we have a split right-exact sequence
    $$
    \begin{tikzcd}
        B_0(H) \ar[r, "B_0(\iota)"] & B_0(H\wr Q) \ar[r, "B_0(\pi)"] & B_0(Q) \ar[r] & 0
    \end{tikzcd}
$$
In other words, $B_0(H\wr Q)\cong B_0(Q)\oplus B_0(H)/\ker B_0(\iota)$.
\end{theorem}

\begin{proof}
        By a result of Kang \cite{Kan14} (see also \cite{Jez16} for a proof not assuming finiteness), $B_0$ commutes with direct sums. The restricted direct sum is a direct limit of its finite direct subsums. As $B_0$ also commutes with direct limits \cite{Mor12}, we get 
    $$B_0(N)\cong \bigoplus_{q\in Q}B_0(H).$$
    As the action of $Q$ shifts the summands regularly, we deduce that $B_0(N)_Q\cong B_0(H)$.

    By Theorem \ref{thm:B0split}, 
    $$B_0(W)\cong B_0(Q)\oplus \frac{K(W,N)}{K(W,N)\cap M_0(W)}.$$
    Denote the second direct summand by $R(W,N)$. By the above, the inclusion $N\hookrightarrow W$ induces an isomorphism $H_2(N,\Z)_Q\cong K(W,N)$. The image of $M_0(N)_Q$ in $H_2(N,\Z)_Q$ is contained in $K(W,N)\cap M_0(W)$, so $R(W,N)$ is a quotient of
    $$H_2(N,\Z)_Q/\im (M_0(N)_Q\to H_2(N,\Z)_Q)\cong B_0(N)_Q\cong B_0(H).$$
    All coordinate inclusions $H\hookrightarrow W$ are conjugate in $W$. Inner automorphisms induce the identity on integral homology, so the surjection $B_0(H)\twoheadrightarrow R(W,N)$ is induced by $\iota$. Also note that $R(W,N)=\ker B_0(\pi)$. This gives the required exactness, and the splitting is induced by the map $B_0(s)$.
\end{proof}
 


\section{Bogomolov multipliers of infinite simple groups}
\label{sec:infinite-simple}

Kunyavski\u{\i} proved that the Bogomolov multiplier of every finite
nonabelian simple group is trivial \cite{Kun08}.  This result
depends essentially on the structure theory of finite simple groups and does
not extend to arbitrary infinite simple groups.  In fact, infinite simple
groups can have nontrivial, and even infinitely generated, homological
Bogomolov multipliers.

Flores and Rodr\'{\i}guez prove that, given an appropriate countable family
$\{H_i\mid i\in I\}$ of countable groups without involutions, there exists a
countable group $G$ with the following properties
\cite[Theorem~5.4]{FloresRodriguez}:
\begin{enumerate}
\item $G$ is infinite and simple;
\item $G$ has the Ol'shanski\u{\i} subgroup property relative to the
      $H_i$: every proper subgroup of $G$ is either cyclic or is contained
      in a conjugate of some $H_i$;
\item the Schur multiplier $H_2(G,\mathbb Z)$ is free abelian of countably
      infinite rank.
\end{enumerate}
The groups $H_i$ may be chosen to be nonabelian simple groups all of whose
proper subgroups are cyclic; suitable groups are supplied by the
Obraztsov--Ol'shanski\u{\i} constructions used in
\cite{FloresRodriguez}.

\begin{theorem}
\label{thm:infinite-simple-nontrivial-b0}
There exists a countable infinite simple group $G$ such that
$B_0 (G)\cong \mathbb Z^{(\mathbb N)}$.
\end{theorem}

\begin{proof}
Choose the family $\{H_i\mid i\in I\}$ in
\cite[Theorem~5.4]{FloresRodriguez} so that every $H_i$ is nonabelian
simple and every proper subgroup of $H_i$ is cyclic.  Let $G$ be the
resulting infinite simple group.

We first show that every abelian subgroup of $G$ is cyclic.  Let $B\leq G$
be abelian.  Since $G$ is nonabelian, $B$ is a proper subgroup of $G$.
The Ol'shanski\u{\i} subgroup property gives two possibilities.  Either
$B$ is cyclic, or
$B\leq H_i^g$
for some $i\in I$ and $g\in G$.  In the latter case $B$ is a proper
subgroup of $H_i^g$, because $H_i^g$ is nonabelian whereas $B$ is
abelian.  By the choice of $H_i$, every proper subgroup of $H_i^g$ is
cyclic.  Thus $B$ is cyclic in either case.
Hence $M_0 (G)=0$.
On the other hand, \cite[Theorem~5.4(3)]{FloresRodriguez} gives
$H_2(G,\mathbb Z)\cong\mathbb Z^{(\mathbb N)}$.
It follows that $B_0 (G)
 =H_2(G,\mathbb Z)/ M_0 (G)
 \cong\mathbb Z^{(\mathbb N)}$.
\end{proof}

The opposite behaviour also occurs among infinite
simple groups.

\begin{example}
\label{ex:infinite-alternating-b0}
Let
$A_\infty=\bigcup_{n\geq5}A_n$
be the group of finitely supported even permutations of a countably
infinite set.
The group $A_\infty$ is the filtered union of the finite alternating
groups $A_n$ under the standard inclusions.  
 Since direct limits commute with homology, it follows that
 $H_2(A_\infty,\mathbb Z)
 \cong\varinjlim_{n}H_2(A_n,\mathbb Z)$.

Let $\alpha\in H_2(A_\infty,\mathbb Z)$.  There exist $n$ and
$\alpha_n\in H_2(A_n,\mathbb Z)$ whose image is $\alpha$.  Since $A_n$
is a finite nonabelian simple group, Kunyavski\u{\i}'s theorem gives
$B_0 (A_n)=0$.
Equivalently,
 $H_2(A_n,\mathbb Z)= M_0 (A_n)$.
Thus $\alpha_n$ is a finite sum of classes represented by commuting
pairs in $A_n$.  Their images remain commuting pairs in $A_\infty$, so
the image $\alpha$ belongs to $ M_0 (A_\infty)$.  Hence
$H_2(A_\infty,\mathbb Z)= M_0 (A_\infty)$,
which proves that $B_0 (A_\infty)=0$.
\end{example}


\section{Finitary Bogomolov multipliers}
\label{s:finitaryB0}

\noindent
Let $G$ be a group and identify the Schur multiplier
$M(G)=H_2(G,\mathbb Z)$ with the kernel of the commutator map
$\kappa_G\colon G\wedge G\rightarrow [G,G]$.
For an inclusion $i_H\colon H\hookrightarrow G$, write
$(i_H)_*: M(H)\to M(G)$ for the induced homomorphism.
Define $M_{F}(G)$ be the sum of all images of the maps $(i_H)_*:H_2(H,\Z)\rightarrow H_2(G,\Z)$, where $H$ runs through all finite subgroups of $G$. Also, let $M_{0F}(G)$ be the sum of all images of the maps $(i_A)_*:H_2(A,\Z)\to H_2(G,\Z)$, as $A$ runs through all finite abelian subgroups of $G$. Put
$$B_{0F}(G)=M_F(G)/M_{0F}(G).$$
This quotient is called the \emph{finitary Bogomolov multiplier} of $G$. Note that if $G$ is a torsion-free group, then $B_{0F}(G)$ is trivial. Also, it is proved in \cite[Corollary 3.4]{Mor12} that if $G$ is a locally finite group, then $B_{0F}(G)\cong B_0(G)$.  In general, the groups  $B_0(G)$ and 
$B_{0F}(G)$ may not be isomorphic, see the one-relator and Burnside-group
examples in \cite{Mor12}.

\begin{lemma}
    \label{lem:wedgesF}
    We have that
    $$M_{0F}(G)= \langle x\wedge y\mid x,y\in G, [x,y]=1,|x|<\infty,|y|<\infty\rangle.$$
\end{lemma}

\begin{proof}
    Let $N$ denote the subgroup on the right.  If $x$ and $y$ commute and
    both have finite order, then $A=\langle x,y\rangle$ is a finite abelian
    subgroup of $G$.  Since $\kappa_A$ is trivial, $M(A)=A\wedge A$, and
    the image of $x\wedge y\in M(A)$ in $G\wedge G$ is the displayed
    element $x\wedge y$.  Thus $N\leq M_{0F}(G)$.

    Conversely, let $A\leq G$ be finite abelian.  The exterior square
    $A\wedge A=M(A)$ is generated by the elements $a\wedge b$ with
    $a,b\in A$.  Their images in $G\wedge G$ belong to $N$, because
    $a$ and $b$ commute and have finite order.  Hence
    $(i_A)_*M(A)\leq N$.  This
   gives $M_{0F}(G)\leq N$.
\end{proof}

In a finitely generated nilpotent group $G$, the torsion elements form a finite subgroup $T$, and every finite subgroup of $G$ is contained in $T$. Thus we have the following:

\begin{proposition}
    \label{prop:BoFNil}
    Let $G$ be a finitely generated nilpotent group. then $M_F(G)=(i_T)_*M(T)$ and $M_{0F}(G)=\langle x\wedge y\mid x,y\in T,[x,y]=1\rangle$ (here $x\wedge y$ is considered as an element of $G\wedge G$).
\end{proposition}

\begin{proposition}
\label{p:fin}
Let $G$ have finitely many conjugacy classes of finite subgroups. Then $B_{0F}(G)$ is a finite abelian group.
\end{proposition}

\begin{proof}
Suppose $K=g^{-1}Hg$.  Let $c_g\colon H\to K$ be conjugation by $g$.
The two composites $H\xrightarrow{c_g}K\hookrightarrow G$ and
$H\hookrightarrow G\xrightarrow{c_g}G$ agree.  An inner automorphism
of $G$ induces the identity on $H_2(G,\mathbb Z)$.  By naturality,
the images of $M(H)$ and $M(K)$ in $M(G)$ are therefore equal.

Choose representatives $H_1,\ldots,H_r$ of the conjugacy classes of
finite subgroups.  Then
\[
 M_F(G)=\big\langle (i_{H_j})_*M(H_j)\mid 1\leq j\leq r\big\rangle.
\]
For every finite group $H_j$, the Schur multiplier $M(H_j)$ is finite.
Thus $M_F(G)$, an abelian group generated by finitely many finite
subgroups, is finite. The result follows.
\end{proof}

The above result in particular applies to polycyclic-by-finite or arithmetic groups \cite{Seg05}, and, for example, to finitely generated Kleinian groups \cite{Fei91}. On the other hand, 
the group $B_{0F}(G)$ can be infinite in general, as the following example shows.
\begin{example}
\label{ex:infin}
    The group $B_{0F}(G)$ can be infinite.  Let
$P$ be the 149th group of order 64 in GAP's library of small groups.  Then
$\B_0(P)\cong C_2$. Take $G$ to be the restricted direct product
    $$
   G=\bigoplus_{n\geq 1}P_n,
   \qquad P_n\cong P.$$
The group $G$ is locally finite.  For finite direct products we have
$\B_0(X\times Y)
 \cong B_0(X)\oplus B_0(Y)$, see Kang \cite[Theorem 1.4]{Kan14} for finite groups, or Jezernik \cite[Theorem 3.5]{Jez16} for general groups.
Homology and $B_0$ commute with direct limits
\cite[Proposition~3.6]{Mor12}; hence, using also
\cite[Corollary~3.4]{Mor12},
$$
 B_{0F}(G)
 \cong  B_0(G)
 \cong \varinjlim_m B_0(P^m)
 \cong \bigoplus_{n\geq1}C_2.
$$
Thus $B_{0F}(G)$ is an infinite elementary abelian $2$-group.
\end{example}

Now let $G$ be a polycyclic group.  Such a group has only finitely many
conjugacy classes of finite subgroups, and the package \texttt{Polycyclic}
can compute representatives $H_1,\ldots,H_r$.  Compute the exterior
square $G\wedge G$, its crossed pairing
$\lambda_G\colon G\times G\to G\wedge G$,
$(x,y)\longmapsto x\wedge y$,
and $\kappa_G\colon G\wedge G\to [G,G]$ using the algorithm of
\cite{Eic08}.  For each $H_i$, compute the analogous data
$H_i\wedge H_i$, $\lambda_{H_i}$, and $\kappa_{H_i}$.  Functoriality of
the exterior square gives a homomorphism
$\phi_i\colon H_i\wedge H_i\longrightarrow G\wedge G$.
If $h_1,\ldots,h_s$ is a polycyclic generating sequence for $H_i$, the
signed ordered wedges
$$
 \lambda_{H_i}(h_a^{\varepsilon},h_b^{\delta}),
 \qquad 1\leq a,b\leq s,\quad \varepsilon,\delta\in\{-1,1\},
$$
generate $H_i\wedge H_i$: this follows by passing to the exterior
quotient from the tensor-square generating theorem of \cite{Bly09}.
Since $M(H_i)=\ker\kappa_{H_i}$,
$$
 M_F(G)=\big\langle\phi_i(\ker\kappa_{H_i})\mid1\leq i\leq r\big\rangle
$$
and
$$
 M_{0F}(G)=\big\langle\phi_i(\ker\kappa_{H_i})\mid
 1\leq i\leq r,\ H_i\text{ abelian}\big\rangle.
$$
Taking the quotient computes $B_{0F}(G)$.

\begin{example}
\label{ex:cryst}
    A crystallographic group $G$ has a normal translation lattice
$L\cong\mathbb Z^n$ and finite point group $P=G/L\leq
\mathrm{GL}_n(\mathbb Z)$.  The action is faithful.  The extension $G$
is polycyclic when $P$ is polycyclic.

The computation accompanying this section gives the following results.
All one-dimensional crystallographic groups, all $17$ affine types in
dimension $2$, and all $219$ affine types in dimension $3$ in
\texttt{CrystCat} have trivial $B_{0F}$.  (The $219$ affine
types correspond to the usual $230$ three-dimensional space-group types,
with enantiomorphic pairs identified by affine equivalence.)

For a higher-dimensional example, let $P=\operatorname{SmallGroup}(64,149)$.
Its minimal faithful permutation representation has degree $12$ and hence
gives a particular integral permutation-lattice representation
$P\leq\mathrm{GL}_{12}(\mathbb Z)$.  For this fixed integral
representation, \texttt{Cryst} returns $16$ space groups up to conjugacy
by translations.  Exactly one is symmorphic; its finitary multiplier is
$C_2$.  The remaining $15$ groups are nonsymmorphic and have trivial
finitary multiplier.  
\end{example}


\section{Implementations}
\label{s:implement}

\subsection{Torsion-free nilpotent groups of class 2}
\label{ss:implement_tfclass2}

Let $G$ be a finitely generated torsion-free group of class $2$ with $V=G^{\rm ab}$ torsion free. Denote $W=G'$.

Our implementation includes:

\begin{itemize}
    \item Construction of the matrix of the commutator map $\beta:\wedge^2V\twoheadrightarrow W$ from commutators of lifts of a basis of $V$, computation of its integral kernel $L_\beta$. 
    \item Computation of approximations of $B_0(G)$ as follows. We compute $G\wedge G$ and $M(G)$. Then we form a subgroup $S$ generated by the wedges of central generators, and wedges of commuting pairs in a bounded range. The computed quotients $M(G)/S$ admits $B_0(G)$ as its quotient. Triviality of $M(G)/S$ implies $B_0(G)=0$. 
\end{itemize}

\subsection{Virtually abelian polycyclic groups}
\label{ss:implement_virtabpoly}

Let $G$ be a virtually abelian polycyclic group. 
The algorithm for computing $B_0(G)$ has been implemented in GAP using the package
\textsf{Polycyclic}.  Given a virtually abelian pcp group $G$, we first construct a normal torsion-free
abelian subgroup $A$ of finite index from the centre of the Fitting subgroup
and form the finite quotient $Q=G/A$.  From a transversal of $A$ in $G$ it
computes the action matrices of $Q$ on $A$ and the corresponding factor set.
For each $(q,r)\in Q^2$, the commuting lifts in $As_q\times As_r$ are then
obtained by solving an integral linear system, whose solution set is either
empty or an affine lattice.  The package \textsf{Polycyclic} already computes $G\wedge G$, the commutator map
$\kappa\colon G\wedge G\to G'$, and the crossed pairing
$\lambda(g,h)=g\wedge h$.  On each affine lattice the map $\lambda$ has
degree at most two in the lattice parameters, so its values on
$0,e_i,2e_i$ and $e_i+e_j$ generate the values on the entire lattice.
The implementation collects these finitely many commuting wedges, forms
their subgroup $M_0(G)\leq\ker\kappa=H_2(G,\mathbb Z)$, and returns
$B_0(G)=H_2(G,\mathbb Z)/M_0(G)$, 
together with its abelian invariant factors.  Thus the computation is exact
and terminating on the promised class of virtually abelian pcp groups; the
finite sampling is justified by the quadratic finite-difference argument
and is not a heuristic bounded search.

\subsection{Finitary Bogomolov multipliers}
\label{ss:implement_finitary}

The computation of $B_{0F}(G)$ has been implemented in GAP for groups given by a polycyclic presentation, using the package \texttt{Polycyclic}. The function first computes the nonabelian exterior square $G\wedge G$, together with the crossed pairing $\lambda_G(x,y)=x\wedge y$ and the natural epimorphism $\kappa_G\colon G\wedge G\to [G,G]$. It then uses \texttt{FiniteSubgroupClasses} to obtain representatives $H_1,\ldots,H_r$ of the conjugacy classes of finite subgroups of $G$. For each $H_i$, the program computes $H_i\wedge H_i$ and constructs the homomorphism $\phi_i\colon H_i\wedge H_i\to G\wedge G$ induced by inclusion. This map is specified on the generating elements $\lambda_{H_i}(h_a^\varepsilon,h_b^\delta)$, where $(h_1,\ldots,h_s)$ is a polycyclic generating sequence for $H_i$ and $\varepsilon,\delta\in\{-1,1\}$, and the implementation verifies that these elements generate $H_i\wedge H_i$. Since $M(H_i)=\ker\kappa_{H_i}$, the program forms
$$
 M_F(G)=\left\langle\phi_i\bigl(\ker\kappa_{H_i}\bigr)\mid 1\leq i\leq r\right\rangle$$
 and
 $$
 M_{0F}(G)=\left\langle\phi_i\bigl(\ker\kappa_{H_i}\bigr)\mid H_i\ \text{abelian}\right\rangle .
$$
Finally, it returns the abelian invariants of the quotient $M_F(G)/M_{0F}(G)$, together with the intermediate subgroups and maps, so that the computation can be independently inspected.

The crystallographic computations are reproduced by the GAP program \texttt{reproduce\_crystallographic.g}, using the packages \texttt{Polycyclic}, \texttt{Cryst}, and \texttt{CrystCat}. In dimensions at most three, the program first treats the two one-dimensional crystallographic groups and then enumerates the $17$ affine space-group types in dimension $2$ and the $219$ affine types in dimension $3$ contained in \texttt{CrystCat}. Each group is converted to a polycyclic presentation, and its finitary Bogomolov multiplier is computed using \texttt{B0FByExteriorSquare}. For the twelve-dimensional example, the program constructs a minimal faithful permutation representation of $P=\operatorname{SmallGroup}(64,149)$, which has degree $12$, and converts it into an integral permutation-matrix representation $P\leq\operatorname{GL}_{12}(\mathbb Z)$. The command \texttt{SpaceGroupsByPointGroupOnRight} then constructs the $16$ crystallographic extension classes associated with this fixed integral representation. For each class, the program determines whether it is symmorphic and computes its finitary Bogomolov multiplier. The calculation confirms that every crystallographic group in dimensions at most three has trivial multiplier, while among the $16$ twelve-dimensional groups the unique symmorphic group has $B_{0F}\cong\mathbb Z/2\mathbb Z$ and the remaining $15$ groups have trivial multiplier.



\section*{Use of AI tools}

\noindent
The author used OpenAI's ChatGPT and Codex as assistive tools in the
preparation of this manuscript.  They were used for language editing,
LaTeX drafting, preliminary literature searches, checking mathematical
arguments, and developing and testing portions of the accompanying GAP
code.  All mathematical statements, proofs, references, and computational
results were subsequently reviewed and verified by the author, who assumes
full responsibility for the content of the paper.


    \end{document}